\documentclass[12pt]{article}
\usepackage{lscape}

\usepackage{a4wide}
\usepackage{enumerate}

\usepackage{a4wide}
\usepackage[parfill]{parskip}
\usepackage{graphicx}
\usepackage{subfigure}
\usepackage{amssymb}
\usepackage{amsthm}
\usepackage{amsmath}
\usepackage{enumerate}

\usepackage{multirow}
\usepackage{setspace}
\usepackage{float}

\newtheorem{theorem}{Theorem}[section]

\newtheorem{lemma}[theorem]{Lemma}
\newtheorem{claim}[theorem]{Claim}

\newtheorem{conjecture}[theorem]{Conjecture}

\newtheorem{problem}[theorem]{Problem}
\newtheorem*{ryser}{Ryser's Conjecture}

\usepackage[margin=20pt, small, bf]{caption}

\title{Intersecting extremal constructions in Ryser's conjecture for $r$-partite hypergraphs}

\author
{
  Abu-Khazneh, Ahmad\\
  \small Department of Mathematics, \\
  \small London School of Economics, \\
  \small London, United Kingdom\\
  \small Email: \texttt{a.abu-khazneh@lse.ac.uk}
  \and
  Pokrovskiy, Alexey\\
  \small Methods for Discrete Structures, \\
  \small Freie Universit\"at, \\ 
  \small Berlin, Germany.\\ 
  \small Email: \texttt{alja123@gmail.com} 
}

\begin{document}

\maketitle

\begin{abstract}
Ryser's Conjecture states that for any $r$-partite $r$-uniform hypergraph the vertex cover number is at most $r-1$ times the matching number.  This conjecture is only known to be true for $r\leq 3$. For intersecting hypergraphs,  Ryser's Conjecture reduces to saying that the edges of every $r$-partite intersecting hypergraph can be covered by $r-1$ vertices. This special case of the conjecture has only been proven for $r \leq 5$. 
  
It is interesting to study hypergraphs which are extremal in Ryser's Conjecture i.e, those hypergraphs for which the vertex cover number is exactly $r-1$ times the matching number. There are very few known constructions of such graphs. For large $r$ the only known constructions come from projective planes and exist only when $r-1$ is a prime power. Mansour, Song and Yuster studied how few edges a hypergraph which is extremal for Ryser's Conjecture can have. They defined $f(r)$ as the minimum integer so that there exist an $r$-partite intersecting hypergraph $\mathcal{H}$ with $\tau({\mathcal{H}}) = r -1$ and with $f(r)$ edges. They showed that $f(3) = 3, f(4) = 6$, $f(5) = 9$, and $12\leq f(6)\leq 15$.

In this paper we focus on the cases when $r=6$ and $7$. We show that $f(6)=13$ improving previous bounds. We also show that $f(7)\leq 22$, giving the first known extremal hypergraphs for the $r=7$ case of Ryser's Conjecture. These results have been obtained independently by Aharoni, Barat, and Wanless.
\end{abstract}

\section{Introduction}

A \textit{hypergraph} consists of a vertex set $V=V(\mathcal{H})$ and a set $E=E(\mathcal{H})$ of edges, where each edge is a nonempty subset of $\mathcal{V}$. A hypergraph is called \textit{$r$-uniform} if all it's edges have the same cardinality $r$, and is called \textit{$r$-partite} if its vertex set can be partitioned into $r$ parts, and every edge contains precisely one vertex from each part. Thus every $r$-partite hypergraph is also $r$-uniform.

A \textit{matching} of a hypergraph $\mathcal{H}$ is a set of pairwise disjoint edges in $\mathcal{H}$, while the \textit{matching number} $\nu(\mathcal{H})$ of $\mathcal{H}$ is the size of a largest matching of $\mathcal{H}$. A \textit{(vertex) cover} of $\mathcal{H}$ is a subset $\mathcal{W} \subset V(\mathcal{H})$ such that every edge of $\mathcal{H}$ contains at least one vertex of $\mathcal{W}$, and the \textit{covering number} $\tau(\mathcal{H})$ of $\mathcal{H}$ is the size of a smallest cover of $\mathcal{H}$.   

A conjecture due to Ryser (that appeared in his student Henderson's thesis, \cite{hend}) relates the covering number and matching number for $r$-partite hypergraphs:\\

\begin{ryser}
  If $\mathcal{H}$ is a $r$-partite hypergraph then $\tau(\mathcal{H}) \le (r -1)\nu(\mathcal{H})$.
\end{ryser}

Setting $r=2$ in Ryser's Conjecture gives K\"{o}nig's Theorem \cite{koi}, which is equivalent to numerous other min-max theorems in graph theory and combinatorics, among them Hall's Theorem \cite{hall}. A hypergraph generalisation of Hall's Theorem was proved by Aharoni and Haxell in \cite{ahha} using topological methods. Using this theorem, Aharoni proved the $r=3$ case  of Ryser's Conjecture \cite{ahro}. The case $r=3$ is the only general case of Ryser's Conjecture proved to date.

An \textit{intersecting hypergraph} $\mathcal{H}$, is a hypergraph in which every two hyperedges share at least one vertex, or equivalently $\nu(\mathcal{H}) = 1$. In the special case of intersecting hypergraphs, Ryser's Conjecture reduces  to saying that every $r$-partite intersecting hypergraph has a cover of size $r-1$.

Ryser's Conjecture for intersecting hypergraphs was proved for $r = 3$ and $4$ in \cite{duch} and \cite{gyra}, while $r=5$ was proved in \cite{duch} and \cite{tuz_un} ($r=2$ is trivial). 

The focus of this paper is on extremal hypergraphs for the intersecting version of Ryser's Conjecture, i.e.\ those $r$-partite hypergraphs for which the cover number is exactly $r-1$. There are very few known constructions of such graphs. For large $r$ the only known extremal constructions come from projective planes and exist only when $r-1$ is a prime power. These constructions are referred to as \textit{truncated projective plane} hypergraphs and are constructed as follows (we assume familiarity with the axioms of the finite projective plane, details of which can be found in \cite{van_lint}). 

Assume $r-1$ is a prime power and consider the projective plane of order $r$ viewed as an $r$-uniform intersecting hypergraph $\mathcal{H}$. Let $v$ be any vertex in $\mathcal{H}$ and let $E_v$ be the set of $r$ hyperedges in $\mathcal{H}$ that contain $v$. Let $\mathcal{H'}$ be the hypergraph formed from $\mathcal{H}$ by removing the vertex $v$ and the hyperedges $E_v$. Since every hyperedge in $\mathcal{H'}$ intersects each hyperedge in $E_v$ exactly once (which follows from the axioms of finite projective planes), $\mathcal{H'}$ can be viewed as an $r$-partite intersecting hypergraph where each partition consists of the vertices of one of the hyperedges in $E_v$ excluding $v$. Since $\mathcal{H'}$ contains $(r-1)^2$ edges and each vertex in $\mathcal{H'}$ also has degree $(r-1)^2$, it can be seen that $\mathcal{H'}$ requires at least $r-1$ vertices to cover its edges. Finally, since the hyperedges of $\mathcal{H'}$ can be covered by the $r-1$ vertices that make up any of the partitions of $\mathcal{H'}$, it can be seen that $\tau(\mathcal{H'}) = r - 1$. Thus $\mathcal{H'}$ is an extremal hypergraph for Ryser's Conjecture, and is what we shall refer to as the truncated projective plane construction.

From the above recipe, it follows that for each finite projective plane of order $q$ there exist a $(q+1)$-partite intersecting extremal hypergraph for Ryser's Conjecture. Also, since a finite projective plane of order $q$ contains $q^2 + q + 1$ lines, the corresponding extremal hypergraph will  contain  $q^2$ hyperedges. Furthermore, since whenever $q - 1$ is a prime power there is a finite projective plane of order $q$, we have an infinite family of $r'$s for which there is an extremal $r$-partite  hypergraph for Ryser's Conjecture.

Our motivation for researching extremal hypergraphs for the intersecting case of Ryser's Conjecture is mainly due to two reasons. The first reason is that as noted by Mansour, Song and Yuster in their study \cite{mans} of such hypergraphs the truncated projective plane construction is not the ``correct'' extremal construction since it contains more hyperedges than necessary. They defined $f(r)$ as the minimum integer so that there exist an $r$-partite intersecting hypergraph $\mathcal{H}$ with $\tau({\mathcal{H}}) \ge r -1$ and with $f(r)$ edges, then they showed that $f(3) = 3, f(4) = 6$, $f(5) = 9$ and that $12 \le f(6) \le 15$, less than the $(r-1)^2$ hyperedges contained in the respective truncated projective plane construction for each of $r= 3, 4, 5 $ and $ 6$. Furthermore, the truncated construction is not defined for all $r$, so it is interesting to find other constructions of extremal hypergraphs.

Another reason to study extremal hypergraphs in the intersecting case, is a recent study by Haxell, Narins, and Szab\'o of extremal hypergraphs in the $r=3$ case of Ryser's Conjecture \cite{HNS}. Using topological methods, they were able to characterize \emph{all} hypergraphs which are extremal for the $r=3$ case of Ryser's Conjecture (not just intersecting ones). Their main result is that for $r=3$, any extremal hypergraph is formed from a vertex-disjoint union of intersecting extremal hypergraphs as well as some ``extra'' edges. It would be interesting to find similar characterizations for $r>3$. An intermediate step towards this is to better understand intersecting extremal graphs.

The first contribution of this paper is that we improve the aforementioned bound $12 \le f(6) \le 15$ proved in \cite{mans} by showing that:\\

\begin{theorem}
  \label{f_6_is_13}
  $f(6) = 13$
\end{theorem}
This theorem was also proved independently by Aharoni, Barat, and Wanless~\cite{Wanless}.

The second contribution of this paper revolves around the value of $f(7)$ for which no bound is currently known. The truncated projective plane construction doesn't help in this case, since it has been proved that finite projective planes of order $6$ do not exist.

Since $6$ is not a prime power, there exist no finite projective planes of order $6$ that are constructed via a vector space of finite fields. However, this is not sufficient to rule out other constructions of finite projective planes of order $6$. The first published proof that no finite projective planes of order $6$ exist is due to Tarry \cite{tarry} in 1901 through exhaustive enumeration, in his proof related to Euler's conjecture on the non-existence of Graeco-Latin squares of order $6$. The non-existence of finite projective planes of order $6$ also follows from Bruck-Ryser theorem\footnote{We note that there is an extension of Bruck-Ryser Theorem for symmetric block designs known as the Bruck-Ryser-Chowla Theorem} \cite{bruck} which states that:\\

\begin{theorem}
  If a finite projective plane of order $q$ exists and $q$ is congruent to $1$ or $2$ (mod $4$), then $q$ must be the sum of two squares.
\end{theorem}

Our second contribution in this paper presents the first known upperbound for $f(7)$.\\

\begin{theorem}
  \label{r7_ext}
  $f(7) \leq 22$
\end{theorem}

To prove Theorem \ref{r7_ext}, we present a $7$-partite hypergraph with $22$ edges that has a covering number equal to $6$, which was generated by the aid of a computer search. 

In  \cite{Wanless}, Aharoni, Barat, and Wanless also found a $7$-partite intersecting extremal hypergraph for Ryser's Conjecture with $17$ edges. In fact, they were able to show that $f(7) = 17$.

\section{The value of $f(6)$}

To settle the case $f(6)$ we will first show that $f(6) > 12$, by proving that $f(6) \neq 12$ and then combine it with the result $f(6) > 11$ established in \cite{mans}. We will then present a $6$-partite intersecting extremal hypergraph with $13$ edges, which shows that $f(6) = 13$.

For a given hypergraph $\mathcal{H}$ and a vertex $v \in V(\mathcal{H})$, we let $E(v)$ denote the set ${\{e \in E(\mathcal{H}): v \in e\}}$, and we denote the degree of $v$ by ${d(v) = |E(v)|}$. We also use the notation $\Delta(\mathcal{H})$ to denote the maximum degree over all vertices of $\mathcal{H}$. Finally, for two distinct vertices $v$ and $w$ in $\mathcal{H}$, the \emph{co-degree} of $v$ and $w$, denoted by $c(v, w)$, is defined as $|E(v) \cap E(w)|$.

In the rest of the paper we will make use of the following trivial bound on the covering number of an intersecting hypergraph: if $\mathcal{H}$ is an intersecting hypergraph then $\tau(\mathcal{H}) \le \frac{E(\mathcal{H})}{2}$. This bound follows since a cover of size $\frac{E(\mathcal{H})}{2}$ can be established via the greedy algorithm given that every two hyperedges in an intersecting hypergraph intersect in at least one vertex. We will call any cover obtained this way a \emph{greedy} cover of the hypergraph.

% Formal definition of Delta above: ${\Delta(\mathcal{H}) = \smash{\displaystyle\max_{v \in V(\mathcal{H})}}d(v)}$
% Taken out because looks awkward inline as it almost touches the line above, didn't find how to include it and fix the line spacing

\subsection{Proof that $f(6) > 12$}

The strategy we adapt to prove that $f(6) \neq 12$, is first to assume that $\mathcal{H}$ is a $6$-partite extremal hypergraph that contains exactly $12$ edges and then showing via a case-by-case analysis that all possible values of $\Delta(\mathcal{H})$ lead to a contradiction. When $\Delta(\mathcal{H})$ is large it can be shown that a cover $\mathcal{C}$ of $\mathcal{H}$ can be formed such that $|\mathcal{C}| < 5$, contradicting the extremality of $\mathcal{H}$. When $\Delta(\mathcal{H})$ is small it can be shown that some of the hyperedges of $\mathcal{H}$ don't intersect each other contradicting the fact that $\mathcal{H}$ is intersecting.

The case $\Delta(\mathcal{H}) = 4$ turns out to be more difficult to deal with than the other cases, and to settle it we will require some facts concerning the degree structure of intersecting $6$-partite hypergraphs with $8$ hyperedges and a covering number equal to $4$. We will start by proving these facts before presenting the proof of $f(6) > 12$.\\

\begin{lemma}
  \label{fact_8_edges}
  If $\mathcal{H'}$ is an intersecting $6$-partite hypergraph with $8$ hyperedges and $\tau(\mathcal{H'}) = 4$, then $\mathcal{H'}$ contains a vertex of degree $3$ in each partition, and there exists two hyperedges in $\mathcal{H'}$ such that they share at least two vertices of degree $3$ in common.
\end{lemma}

\begin{proof}
  For the rest of proof let $\mathcal{H'}$ be as in the statement of the Lemma. We can assume $\Delta(\mathcal H')\leq 3$, otherwise we can find a cover $\mathcal{C}$ of $\mathcal{H'}$ with $|\mathcal{C}| \leq 3$ by including in $\mathcal{C}$ a vertex of degree more than $3$, and greedily covering the remaining uncovered hyperedges. We will proceed via a series of claims.

  \begin{claim}
    \label{7Edge}
    Every $6$-partite, intersecting hypergraph $\mathcal{G}$ with $7$ edges and satisfying $\Delta(\mathcal{G})\leq 3$ has at least $2$ vertices of degree $3$.
  \end{claim}
  \begin{proof}
    Suppose, for the sake of contradiction that $\mathcal{G}$ contains at most one vertex of degree~$3$. Let $v$ be this vertex (if it exists).  Since $\mathcal{G}$ is intersecting, there are $\binom {7}{2}=21$ intersections between the edges.  Three of these intersections can occur at $v$, and the rest must all occur at distinct vertices of degree $2$.  Therefore there must be at least 19 vertices in $\mathcal{G'}$ of degree~$\geq 2$.
    By the Pigeonhole Principle some partition of $\mathcal{G}$ has at least $4$ vertices of degree at least $2$.  Since $\mathcal{G}$  has $7$ edges, some edge must pass through two vertices in this partition contradicting $\mathcal{G}$ being $6$-partite.
  \end{proof}

  \begin{claim}
    \label{1Deg3}
    Every edge in $\mathcal{H'}$ contains a vertex of degree $3$.
  \end{claim}

  \begin{proof}
    If $E$ is an edge of $\mathcal{H'}$, then it has $6$ vertices and must intersect the $7$ other edges of $\mathcal{H'}$.  By the Pigeonhole Principle, one of the vertices of $E$ must have degree $3$.
  \end{proof}

  \begin{claim}
    \label{2Deg3}
    For any pair of vertices $u$ and $v$ of degree $3$ in $\mathcal{H'}$, $c(u, v) \geq 1$. 
  \end{claim}

  \begin{proof}
    Suppose that there are two vertices $u, v \in V(\mathcal{H'})$ of degree $3$ which are not contained in a common edge.  Then, since $|E(\mathcal{H'})|=8$, there are only two edges in $\mathcal{H'}$ which do not contain either $u$ or $v$.  These two edges must intersect in some vertex $w$.  This gives a cover $\{u,v,w\}$ of $\mathcal{H'}$ of order $3$, contradicting our assumption that $\tau(\mathcal{H'}) > 3$.
  \end{proof}

    Let $\mathcal{K}$ be the non-uniform hypergraph formed from $\mathcal{H'}$ by deleting the vertices with degree less than $3$.  Formally $V(\mathcal{K})$ is the set of vertices of $\mathcal{H'}$ with degree $3$ and $E(\mathcal{K})=\{A\cap V(\mathcal{K}): A\in \mathcal{H'}\}$.  We allow $\mathcal{K}$ to have repeated edges in the case when $A\cap V(\mathcal{K})= A'\cap V(\mathcal{K})$ for  distinct edges  $A, A'\in \mathcal{H'}$.

  Notice that by Claim \ref{1Deg3}, we have that $|\mathcal{K}|=|\mathcal{H'}|=8$ and the edges in $\mathcal{H}$ have order at least $1$.  Moreover, from the definition of $\mathcal{K}$, we have that $\mathcal{K}$ satisfies the conclusion of Claim \ref{2Deg3} and $\mathcal{K}$ is 3-regular.\\

  \begin{claim}\label{SmallEdgesInK}
    Let $A$ be an edge of $\mathcal{K}$.  We have that $|A|\leq |V(\mathcal{K})|-2$.
  \end{claim}

  \begin{proof}
    By the definition of $\mathcal{K}$, there is an edge $A'\in \mathcal{H'}$ satisfying $A=A'\cap V(\mathcal{K})$.  Let $\mathcal{H''}$ be the hypergraph formed from $\mathcal{H'}$ by removing the edge $A'$.  It is easy to check that $\mathcal{H''}$ satisfies all the conditions of Claim~\ref{7Edge}, and hence contains two vertices $u$ and $v$ with degree $3$.  Since $\Delta(\mathcal H')\leq 3$, the vertices $u$ and $v$ could not be contained in $A'$ (or $A$) giving the result.
  \end{proof}

  \begin{claim}
    \label{Total3}
    $|V(\mathcal{K})| = 6$.
  \end{claim}

  \begin{proof}
    Since  all edges in $\mathcal{K}$ contain at least one vertex, Claim~\ref{SmallEdgesInK} implies that  $|V(\mathcal{K})|\geq 3$.    
    
    Suppose that $|V(\mathcal{K})|= 3$. By Claim~\ref{SmallEdgesInK}, we have that $|E|\leq 1$ for every edge $E\in \mathcal{K}$.  This contradicts $\mathcal{K}$ satisfying Claim~\ref{2Deg3}.

    Suppose that $|V(\mathcal{K})|= 4$. As in the previous case, Claim~\ref{SmallEdgesInK} implies that we have $|E|\leq 2$ for every edge $E\in \mathcal K$.  Then, Claim~\ref{2Deg3} implies that for every pair of distinct vertices $u, v\in V(\mathcal{K})$ the edge $\{u,v\}$ is in $\mathcal{K}$.  Since $\mathcal{K}$ is $3$-regular, there cannot be any other edges in $\mathcal{K}$, which contradicts $\mathcal{K}=8$.

    Suppose that $|V(\mathcal{K})|= 5$. Claim \ref{SmallEdgesInK} implies that we have $|E|\leq 3$ for every edge $E\in \mathcal(K)$. Let $e_i$ be the number of edges $E\in \mathcal{K}$ satisfying $|E|=i$. Notice that since $|E|\leq 3$ for every edge $E\in \mathcal{K}$, we have that $e_i=0$ for $i>3$.

    Since $\mathcal{K}$ has $5$ vertices and $8$ edges and is $3$-regular, we have the following.

    \begin{align}
      e_1+e_2+e_3&=|\mathcal{K}|=8, \label{EdgeSum}\\
      3e_3+2e_2+e_1&=3|V(\mathcal{K})|=15. \label{VertexSum}
    \end{align}
    
    Combining  (\ref{EdgeSum}) and (\ref{VertexSum}), we obtain the following

    \begin{align}
      e_3&=e_1-1, \label{e11}\\
      e_2&=9-2e_1,\label{e12}
    \end{align}
    
    There are five cases, depending on the value of $e_1$.
    
    \begin{itemize}
    \item Suppose that $e_1\leq 1$. Then (\ref{e11}), together with $e_3\geq 0$ implies that in fact $e_1=1$ and hence from (\ref{e11}) and (\ref{e12}) we obtain  $e_2=7$ and $e_3=0$.   This contradicts Claim~\ref{2Deg3} which implies that $e_2+3e_3\geq \binom{5}{2}=10$.
    \item Suppose that $e_1= 2$.  Then we have $e_3= 1$ and $e_2 = 5$.  Again, this contradicts $e_2+3e_3\geq \binom{5}{2}=10$.
    \item Suppose that $e_1= 3$.  Then we have $e_3= 2$ and $e_2 = 3$.  Let $\{v_1\}$, $\{v_2\}$, and $\{v_3\}$ be the three edges of $\mathcal{K}$ of order $1$.  Notice that by  Claim~\ref{2Deg3} and $\Delta(\mathcal{K})\leq 3$, for each $i$, the vertex  $v_i$ must each be contained in two edges $E$, $F$ of order $3$ satisfying $E\cap F=\{v_i\}$.  This leads to a contradiction since there are only two edges in $\mathcal{K}$ of order 3.
    \item Suppose that $e_1= 4$.  Then we have $e_3= 3$ and $e_2 = 1$.  Let $\{v_1, v_2\}$ be the edge of order 2 in $\mathcal{K}$.  Since $|V(\mathcal{K})|=5$ and there are four edges of $\mathcal{K}$ of order 1, either $\{v_1\}$ or $\{v_2\}$ must be an edge of $\mathcal{K}$. 
There can only be one more edge going through this vertex, and by Claim~\ref{2Deg3}, it would also have to pass through the remaining three vertices $v_3, v_4$, and $v_5$. This contradicts $|E|\leq 3$ holding for every edge in $\mathcal K$.
    \item Suppose that $e_1\geq 4$.  In this case (\ref{e12}) gives $|e_2|<0$ which is impossible.
    \end{itemize}
  \end{proof}

  \begin{claim}
    \label{pairs}
    The hypergraph $\mathcal{K}$ contains two edges $E$ and $F$ such that $E\cap F\geq 2$.

  \end{claim}
  \begin{proof}

    Claim \ref{SmallEdgesInK} implies that we have $|E|\leq 4$ for any edge $E\in \mathcal{K}$.

    Suppose that we have an edge $E$ of order $4$ in $\mathcal{K}$.  Let $E=\{v_1,v_2,v_3,v_4\}$.  Since $\mathcal{K}$ is $3$-regular each vertex $v_i$ is contained in two edges $F_i^1$ and $F_i^2$ other than $E$.  Since $\mathcal{K}$ has 8 edges, $F_i^a=F_j^b$ for some $i\neq j$.  Therefore we have $\{v_i,v_j\}\subseteq F_i^a\cap E$ implying the claim.

    Suppose that all edges $E \in \mathcal{K}$ satisfy $|E|\leq 3$.  If a vertex $v\in V(\mathcal{K})$ is contained in three edges of order $3$, then two of these edges have intersection greater than $2$, proving the claim.  Therefore we have that any $v\in V(\mathcal{K})$  is contained in at most two edges of order $3$.  By Claim~\ref{2Deg3}, every vertex $v\in V(\mathcal{K})$ is then contained in exactly two edges of order $3$ and one edge of order $2$.  The number of edges of order $2$ in $\mathcal{K}$ must therefore be $|V(\mathcal{K})|/2=3$ and the number of edges of order $3$ in $\mathcal{K}$ must be $2|V(\mathcal{K})|/3=4$. This contradicts $\mathcal{K}$ having $8$ edges.
  \end{proof}

  Now Claim~\ref{Total3} proves that $\mathcal{H'}$ contains six vertices of degree $3$, and Claim \ref{2Deg3} shows that these vertices are all in different partitions of $\mathcal{H}$. Claim~\ref{pairs} shows that there exist at least two hyperedges in $\mathcal{H'}$ such that they share at least two vertices of degree $3$ in common. Together these facts prove Lemma \ref{fact_8_edges}.
  \end{proof}

  Using Lemma~\ref{fact_8_edges} we are able to determine precisely all possible degree structures of intersecting $6$-partite hypergraphs with $8$ hyperedges and a covering number equal to $4$.\\

  \begin{lemma}
  \label{degs}
  If $\mathcal{H'}$ is an intersecting $6$-partite hypergraph with $8$ hyperedges and $\tau(\mathcal{H'}) = 4$, then $\mathcal{H'}$ has one of the following degree structure:
  \begin{itemize}
  \item In all $6$ partitions of $\mathcal{H'}$, each partition contains one vertex of degree $3$, two vertices of degree $2$ and one vertex of degree $1$, or
  \item In $5$ partitions of $\mathcal{H'}$ it contains one vertex of degree $3$, two vertices of degree $2$ and one vertex of degree $1$, and in the $6$th partition it contains one vertex of degree $3$, one vertex of degree $2$, and four vertices of degree $1$.
  \end{itemize}

\end{lemma}

\begin{proof}

  Since $\mathcal{H'}$ is an intersecting hypergraph that contains $8$ hyperedges, the number of intersections between the hyperedges of $\mathcal{H'}$ is at least $\binom{8}{2} = 28$. From Lemma~\ref{fact_8_edges} we also know that $\Delta(\mathcal{H'}) = 3$ and that $\mathcal{H'}$ contains six vertices of degree $3$. Since each vertex of degree $3$ contributes $3$ intersections between the hyperedges of $\mathcal{H'}$, the maximum number of intersection contributed by the vertices of degree $3$ is $18$. 

  However, by Lemma~\ref{fact_8_edges}, we know that at least one pair of hyperedges have in common at least two vertices of degree $3$, therefore we can reduce the previous bound by $1$ to account for this duplication, which makes the maximum number of intersection contributed by the vertices of degree $3$ equal to $17$. Hence, the vertices of degree $2$ in $\mathcal{H'}$ need to account for at least $28 - 17 = 11$ of the intersections in $\mathcal{H'}$.

  Since $|E(\mathcal{H'})| = 8$, and each partition of $\mathcal{H'}$ contains a vertex of degree $3$, the maximum number of degree $2$ vertices that $\mathcal{H'}$ can contain in each partition is two. Therefore if $\mathcal{H'}$ contains $11$ vertices of degree $2$ then by the Pigeonhole Principle in at least five partitions of $\mathcal{H'}$ it will contain two vertices of degree $2$, and in the remaining partition we must have either one vertex of degree $2$ or two vertices of degree $2$.

  If one of the partitions of $\mathcal{H'}$ contains exactly one vertex of degree $2$, then apart from the vertex of degree $3$ the remaining vertices in that partition will all have degree $1$. These two possibilities prove the degree scheme stated in Lemma~\ref{degs}.
\end{proof}

\begin{lemma}
  \label{f_6_is_not_12}
  $f(6) \neq 12$
\end{lemma}

\begin{proof}
  Let $\mathcal{H}$ be a $6$-partite intersecting hypergraph containing $12$ edges and assume that $\tau(\mathcal{H}) = 5$. We will proceed by showing that all possible values of $\Delta(\mathcal{H})$ lead to a contradiction.

  \begin{description}

  \item[Case $\Delta(\mathcal{H}) \ge 6$:] Assume that $\Delta(\mathcal{H}) \ge 6$, and let $v \in V(\mathcal{H})$ be vertex such that $d(v) \ge 6$, finally denote by $\mathcal{H'} \subset E(\mathcal{H})$ the set of hyperedges that don't contain $v$, which forms an intersecting $6$-partite sub-hypergraph of $\mathcal{H}$.
    
    We have $|E(\mathcal{H'})| \le 6$ and therefore we can greedily cover $\mathcal{H'}$ with a cover $\mathcal{C}$ such that $|\mathcal{C}| \le 3$. Therefore the set $\mathcal{C'} = \mathcal{C} \cup \{v\}$ covers $\mathcal{H}$, and $|\mathcal{C}'| < 5$ which contradicts $\mathcal{H}$ being extremal.

  \item[Case $\Delta(\mathcal{H}) = 5$:] Assume that $\Delta(\mathcal{H}) = 5$ and let $v \in V(\mathcal{H})$ such that $d(v) = 5$, and define the intersecting $6$-partite sub-hypergraph $\mathcal{H'} \subset E(\mathcal{H})$ to consist of the $7$ hyperedges in $E(\mathcal{H})$ that don't contain $v$.
    
    If $\mathcal{H}'$ has a cover $C$ such that $|C| \le 3$, then the cover $\mathcal{C'} = \mathcal{C} \cup \{v\}$ covers $\mathcal{H}$ and $|\mathcal{C'}| < 5$ which contradicts $\mathcal{H}$ being extremal. We can therefore assume that $\tau(\mathcal{H'}) = 4$.
    
    If any $3$ or more hyperedges of $\mathcal{H}'$ intersect in a vertex $v'$, then we can greedily cover the remaining hyperedges of $\mathcal{H}'$ by $2$ vertices or less, contradicting $\tau(\mathcal{H'}) = 4$. Therefore, we can suppose that $\Delta(\mathcal{H'}) \le 2$.

    However, if $\Delta(\mathcal{H'}) \le 2$ then the maximum number of intersections that can occur in a partition of $\mathcal{H'}$ is $3$ intersections, which occurs when a partition of $\mathcal{H'}$ contains three vertices of degree $2$. It follows that the maximum number of intersections in all of $\mathcal{H'}$ is equal to $18$. However, we require at least $\binom{7}{2} = 21$ intersections between hyperedges of $\mathcal{H'}$, for $\mathcal{H'}$ to be an intersecting hypergraph, which leads to a contradiction. 

  \item[Case $\Delta(\mathcal{H}) = 4$:] Assume that $\Delta(\mathcal{H}) = 4$ and let $v \in V(\mathcal{H})$ be a vertex such that $d(v) = 4$. Let $\mathcal{H'}$ be the intersecting $6$-partite sub-hypergraph $\mathcal{H'} \subset E(\mathcal{H})$ consisting of the $8$ hyperedges in $E(\mathcal{H})$ that don't contain $v$.

    If we can cover $\mathcal{H'}$ by a cover $\mathcal{C}$ such that $|\mathcal{C}| \le 3$, then the set $\mathcal{C'} = \mathcal{C} \cup \{v\}$ covers the whole of $\mathcal{H}$, and since $|\mathcal{C'}| < 5$ this will contradict $\mathcal{H}$ being extremal. Therefore can assume that $\tau(\mathcal{H'}) = 4$.

        Since $\tau(\mathcal{H'}) = 4$, then as in the proof of Lemma~\ref{fact_8_edges}, we must have $\Delta(\mathcal{H'}) \le 3$ (since otherwise, we could cover $4$ edges by one vertex, and the remaining edges greedily by $2$ vertices.). 

    Denote by $\mathcal{H''}$ the set of four hyperedge that contain the vertex $v'$ of degree $4$ (i.e. the hyperedges not in $\mathcal{H'}$). Since $\mathcal{H}$ is an intersecting hypergraph, the number of intersections in $\mathcal{H}$ between hyperedges in $\mathcal{H''}$ and hyperedges in $\mathcal{H'}$ is equal to $4 \cdot 8 = 32$, and these intersections need to occur in $5$ partitions of $\mathcal{H}$; since in the partition that contain $v'$ the hyperedges in $\mathcal{H''}$ are disjoint from the hyperedges in $\mathcal{H'}$.

    From Lemma \ref{degs} we know that $\mathcal{H'}$ can have two types of degree schemes in its partitions, which we will refer to as \textit{Type A} and \textit{Type B}:

    \begin{description}
    \item[Type A:] Partitions that have \textit{Type A} contain one vertex of degree $3$, two vertices of degree $2$ and one vertex of degree $1$,
    \item[Type B:] Partitions that have \textit{Type B} contain one vertex of degree $3$, one vertex of degree $2$ and four vertices of degree $1$. 
    \end{description}

    We will now establish the maximum number of intersections possible that can occur between the hyperedges of $\mathcal{H''}$ and the hyperedges of $\mathcal{H'}$ in each of the two types of degree schemes and show that this is less the minimum required for $\mathcal{H}$ to be intersecting. \\
    
    %We will now establish the maximum number of intersections possible that can occur between the hyperedges of $\mathcal{H''}$ and the hyperedges of $\mathcal{H'}$ in each of the two types of degree schemes. We will be able to use these bounds to show that the total number of possible intersections in $\mathcal{H}$ between hyperedges in $\mathcal{H'}$ and hyperedges in $\mathcal{H''}$ is less than the minimum required for $\mathcal{H}$ to be intersecting, which allows us to obtain a contradiction. \\

    \begin{claim}
      Let $S$ be a partition of $\mathcal{H'}$  of \textit{Type A}. Then the maximum number of intersection in $\mathcal{H}$ that can occur between hyperedges in $\mathcal{H''}$ and hyperedges in $\mathcal{H'}$ within $S$ is at most $6$. 
    \end{claim}

    \begin{proof}
      If all the hyperedges in $\mathcal{H''}$ contained a vertex from $S$, then $S$ would cover all of $\mathcal{H}$ and $|S| = 4$, contradicting the fact that $\mathcal{H}$ is extremal. Thus at most three hyperedges in $\mathcal{H''}$ can contain a vertex from $S$.

      Let $w'$ be the vertex in $S$ that has degree $3$. We note that if more than one hyperedge from $\mathcal{H''}$ contained $w'$, then $w'$ will have a degree in $\mathcal{H}$ that exceeds $4$, which contradicts $\Delta(\mathcal{H}) = 4$. Therefore at most one hyperedge of $\mathcal{H''}$ can contain $w'$. 

      Suppose that at most two hyperedges of $\mathcal{H''}$ contain vertices from $S$, since at most one of them can contain a vertex of degree $3$, this case trivially satisfies the claim.

      Thus the only remaining case that needs to be checked is when all three hyperedges of $\mathcal{H''}$ contain a vertex from $S$.

      Let $e_i$ be the number of hyperedges in $\mathcal{H''}$ that contain a vertex in $S$ of degree $i$ in $\mathcal{H'}$. From the above we have: 

      \setcounter{equation}{0}
      \begin{align}
        e_1 + e_2 + e_3 \le 3 \label{12and3}\\
        e_1 \le 3 \\
        e_2 \le 3 \\
        e_3 \le 1
      \end{align}

      Suppose that exactly three hyperedges of $\mathcal{H''}$ contain vertices from $S$, and one of the hyperedges in $\mathcal{H''}$ contains $w'$. Let $e$ and $e'$ denote the remaining two hyperedges of $\mathcal{H'}$ that contain a vertex in $S$. It can be seen that $e$ and $e'$ contain the vertices in $S$ of degree $2$ in $\mathcal{H'}$ in three possible ways, and we first show two of these possibilities lead to a contradiction:

      \begin{itemize}
      \item Each of $e$ and $e'$ contain one of the vertices in $S$ of degree $2$ in $\mathcal{H'}$. In this case, $w'$ and the two vertices in $S$ of degree $2$ will cover $4 + 3 + 3 = 10$ hyperedges of $\mathcal{H}$, and since we can cover the remaining two hyperedges of $\mathcal{H}$ by a vertex, this will contradict $\tau(\mathcal{H}) = 5$.
        
      \item Hyperedges $e$ and $e'$ contain the same vertex in $S$ of degree $2$ in $\mathcal{H'}$. In this case the aforementioned vertex and $w'$ cover $4 + 4 = 8$ hyperedges of $\mathcal{H}$. Since, we can greedily cover the remaining $4$ hyperedges of $\mathcal{H}$ with $2$ vertices, this will allow us to cover $\mathcal{H}$ with $4$ vertices, contradicting the fact that $\tau(\mathcal{H}) = 5$.
        
      \item At most one of the hyperedges $e$ and $e'$ contains a vertex in $S$ of degree $2$ in $\mathcal{H'}$.
      \end{itemize}

      From the above case analysis, it follows that if one of the hyperedges in $\mathcal{H''}$ contained the vertex in $S$ of degree $3$ in $\mathcal{H'}$, then at most one hyperedge from $\mathcal{H''}$ contains a vertex in $S$ of degree $2$ in $\mathcal{H'}$. We represent this as the inequality:

      \begin{align}
        e_2 + 2e_3 \le 3 \label{2and3}
      \end{align}

      The number of intersections between hyperedges in $\mathcal{H''}$ and vertices in $S$, can be represented as the inequality $e_1 + 2e_2 + 3e_3$. By combining the inequalities (\ref{12and3}) and (\ref{2and3}) we obtain the following bound on the number of intersections:

      \begin{align}
        e_1 + 2e_2 + 3e_3 \le 6
      \end{align}

      Which proves that the maximum number of intersections between the set $\mathcal{H''}$ and partitions with degree scheme of \textit{Type A} is equal to $6$.

    \end{proof}

    \begin{claim}
      Let $S$ be a partition of $\mathcal{H}$  of \textit{Type B}. Then the maximum number of intersection in $\mathcal{H}$ that can occur between hyperedges in $\mathcal{H''}$ and hyperedges in $\mathcal{H'}$ within $S$ is at most $7$. 
    \end{claim}

    \begin{proof}

      %The difference in this case from the previous one is that partitions of $\mathcal{H'}$ with degree scheme of \textit{Type B} have more vertices than partitions with degree scheme of \textit{Type A}, which we will see allows the maximum intersection between hyperedges in $\mathcal{H''}$ and hyperedges in $\mathcal{H'}$ to increase to $7$. 

      Let $w'$ be the vertex in $S$ of degree $3$ in $\mathcal{H'}$, and let $w''$ be the vertex in $S$ of degree $2$ in $\mathcal{H'}$. We note that the no more than one hyperedge of $\mathcal{H''}$ can contain $w'$, otherwise $w'$ will have a degree that exceeds $4$ in $\mathcal{H}$ which contradicts $\Delta(\mathcal{H}) = 4$. Similarly, $\Delta(\mathcal{H}) = 4$ implies that the maximum number of hyperedges in $\mathcal{H''}$ that can contain $w''$ in $S$ is equal to $2$.

      Let $e_i$ be the number of hyperedges in $\mathcal{H''}$ that contain a vertex in $S$ of degree $i$ in $\mathcal{H'}$. From the above we have: 

      \setcounter{equation}{0}
      \begin{align}
        e_1 + e_2 + e_3 \le 4 \label{B_12and3}\\
        e_1 \le 4 \\
        e_2 \le 2 \\
        e_3 \le 1
      \end{align}

      If a hyperedge in $\mathcal{H''}$ contains $w'$, and more than one hyperedge in $\mathcal{H''}$ contain $w''$, then $w'$ and $w''$ cover $8$ or more hyperedges of $\mathcal{H}$, and therefore the remaining hyperedges can be greedily covered by two vertices or less, contradicting $\tau(\mathcal{H}) = 5$. Thus if one of the hyperedges in $\mathcal{H''}$ contains $w'$, then at most one other hyperedge of $E'$ can contain $w''$, or in inequality form:

      \begin{align}
        e_2 + 2e_3 \le 3 \label{B_2and3}
      \end{align}

      We have that the expression $e_1 + 2e_2 + 3e_3$ represents number of intersection between $\mathcal{H''}$ and $\mathcal(H')$, which we can bound by combining the inequalities (\ref{B_12and3}) and (\ref{B_2and3}) we obtain:

      \begin{align}
        e_1 + 2e_2 + 3e_3 \le 7
      \end{align}

      Which proves that the maximum number of intersections between the set of vertices $\mathcal{H''}$ and $\mathcal{H'}$ in a partition with degree scheme of Type B is equal to $7$.

    \end{proof}
    
    Since there is only one partition with degree scheme of \textit{Type B}, and all intersections between $\mathcal{H''}$ and $\mathcal{H'}$ occur in five partitions of $\mathcal{H'}$ then the maximum number of intersection that can occur between $\mathcal{H''}$ and $\mathcal{H'}$ is equal to $7 + 6\cdot4 = 31$, which is one short of the $32$ intersections required to make $\mathcal{H}$ intersecting, a contradiction.
    
  \item[Case $\Delta(\mathcal{H}) \le 3 $:] Since $\mathcal{H}$ is extremal each partition needs to have at least 5 vertices (otherwise the vertices of partition with less than $5$ vertices will form a cover of $\mathcal{H}$ contradicting $\tau(H) = 5$), therefore each partition can have at most three vertices with degree $3$.

    Hence the maximum number of intersections between the hyperedges that can occur in a particular partition of $\mathcal{H}$ is when the partition consists of three vertices with degree $3$, along with another vertex of degree $2$ and another vertex of degree $1$, in which case the maximum number of intersections per partition would be equal to $10$. It follows that the maximum total number of intersections that can occur in all the partitions of $\mathcal{H}$ is $60$. 

    However, if $\mathcal{H}$ is an intersecting hypergraph with $12$ edges then it will need to have $\binom{12}{2} = 66$ intersections. Therefore a hypergraph with $\Delta(\mathcal{H}) \le 3$ can't be extremal.
  \end{description}
\end{proof}

\subsection{An example showing $f(6) = 13$}
    
    In this section we present a $6$-partite intersecting hypergraph $\mathcal{H}$ such that $\tau(\mathcal{H}) = 5$. All partitions of $\mathcal{H}$ except the first one contain $5$ vertices, while the first partition contains $6$ vertices. We denote the hyperedges of $\mathcal{H}$ by $E_i$ for $1 \le i \le 13$. We will use the notation $(i,j)$ to denote the $j$-th vertex in the $i$-th partition of $\mathcal{H}$. The hyperedges of $\mathcal{H}$ are:

    \begin{center}
      \doublespacing
      {\small
      $E_1 =\{(1,1), (2,4), (3,4), (4,5), (5,3), (6,5)\},$\hspace{.05\textwidth}$E_2 =\{(1,2), (2,5), (3,2), (4,5), (5,5), (6,3)\},$\\
      $E_3 =\{(1,3), (2,4), (3,5), (4,3), (5,4), (6,3)\},$\hspace{.05\textwidth}$E_4 =\{(1,4), (2,1), (3,5), (4,4), (5,5), (6,5)\},$\\
      $E_5 =\{(1,4), (2,5), (3,4), (4,2), (5,4), (6,4)\},$\hspace{.05\textwidth}$E_6 =\{(1,5), (2,2), (3,5), (4,5), (5,1), (6,4)\},$\\
      $E_7 =\{(1,5), (2,5), (3,1), (4,3), (5,2), (6,5)\},$\hspace{.05\textwidth}$E_8 =\{(1,5), (2,4), (3,3), (4,2), (5,5), (6,2)\},$\\
      $E_9 =\{(1,5), (2,3), (3,4), (4,4), (5,3), (6,3)\},$\hspace{.05\textwidth}$E_{10} =\{(1,6), (2,2), (3,4), (4,3), (5,5), (6,1)\},$\\
      $E_{11} =\{(1,6), (2,4), (3,2), (4,4), (5,2), (6,4)\},$\hspace{.05\textwidth}$E_{12} =\{(1,6), (2,5), (3,5), (4,1), (5,3), (6,2)\},$\\
      $E_{13} =\{(1,6), (2,3), (3,3), (4,5), (5,4), (6,5)\}$\\            
      }
    \end{center}

    In the appendix we provide another representation of $\mathcal{H}$ which presents it in terms of its degree structure. It is easy to check that $\mathcal{H}$ is intersecting, and by noting that five of the six partitions of $\mathcal{H}$ contain $5$ vertices we see that $\tau(\mathcal{H}) \leq 5$. \\
    
    \begin{lemma}
      \label{6_is_extremal}
      $\tau(\mathcal{H}) = 5$
    \end{lemma}
      
    It could be verified using a computer search that $\mathcal{H}$ cannot be covered by less than five vertices by enumerating all possible subset of $V(\mathcal{H})$ consisting of four vertices and checking if they cover $\mathcal{H}$. Since $|V(\mathcal{H})| = 31$ this can be executed very quickly on a standard desktop computer.

    However by making some observations on the degree and intersection structure of $\mathcal{H}$ we are able to present a proof in the appendix that $\tau(\mathcal{H}) > 4$ by checking far fewer cases in comparison to the $31 \choose 4$ cases checked by the total enumeration approach.

    Lemma~\ref{6_is_extremal} allows us to complete the prove of Theorem~\ref{f_6_is_13}.

    \begin{proof}[Proof of Theorem~\ref{f_6_is_13}]
      From \cite{mans} we know that $f(6) > 11$, and by Theorem~\ref{f_6_is_not_12} we know that $f(6) \neq 12$. Therefore, by Lemma~\ref{6_is_extremal} we have that $f(6)$ must be equal to $13$.
    \end{proof}

\section{The case of $f(7)$}

By the aid of a computer program we were able to generate a $7$-partite hypergraph $\mathcal{H'}$ that contains $22$ hyperedges and has a covering number of $6$ vertices. 

Again, using the notation $(i,j)$ to denote the $j$-th vertex in the $i$-th partition of $\mathcal{H'}$. Using this notation the hyperedges of $\mathcal{H'}$ are

\begin{center}

  \doublespacing
  {\small

      $E_1  =\{ (1 , 1),  ( 2 , 1 ),  ( 3 , 1 ),  ( 4 , 1 ),  ( 5 , 1 ),  ( 6 , 1 ),  ( 7 , 1 ) \},$  $E_2  =\{ (1 , 1) , ( 2 , 2 ),  ( 3 , 2 ),  ( 4 , 2 ),  ( 5 , 2 ),  ( 6 , 3 ),  ( 7 , 3 )\},$\\ 

      $E_3  =\{ (1 , 1 ),  ( 2 , 3 ),  ( 3 , 3 ),  ( 4 , 3 ),  ( 5 , 3 ),  ( 6 , 4 ),  ( 7 , 4 ) \},$ \hspace{.05\textwidth} $E_4  =\{ (1 , 1),  ( 2 , 4 ),  ( 3 , 4 ),  ( 4 , 4 ),  ( 5 , 4 ),  ( 6 , 5 ),  ( 7 , 5 ) \},$ \\ 

      $E_5  =\{ ( 1 , 2 ),  ( 2 , 1 ),  ( 3 , 2 ),  ( 4 , 3 ),  ( 5 , 4 ),  ( 6 , 6 ),  ( 7 , 6 ) \},$\hspace{.05\textwidth}$E_6  =\{ ( 1 , 3 ),  ( 2 , 1 ),  ( 3 , 2 ),  ( 4 , 5 ),  ( 5 , 5 ),  ( 6 , 4 ),  ( 7 , 5 ) \},$\\ 

$E_7  =\{ ( 1 , 5 ),  ( 2 , 3 ),  ( 3 , 2 ),  ( 4 , 6 ),  ( 5 , 1 ),  ( 6 , 5 ),  ( 7 , 2 )\},$\hspace{.05\textwidth}$E_8  =\{ ( 1 , 4 ),  ( 2 , 2 ),  ( 3 , 6 ),  ( 4 , 1 ),  ( 5 , 4 ),  ( 6 , 4 ),  ( 7 , 2 ) \},$\\ 

$E_9  =\{ ( 1 , 3 ),  ( 2 , 5 ),  ( 3 , 3 ),  ( 4 , 1 ),  ( 5 , 2 ),  ( 6 , 5 ),  ( 7 , 6 ) \},$\hspace{.05\textwidth} $E_{10}  =\{ ( 1 , 3 ),  ( 2 , 6 ),  ( 3 , 4 ),  ( 4 , 3 ),  ( 5 , 2 ),  ( 6 , 1 ),  ( 7 , 2 ) \},$
\\ 

$E_{11}  =\{ ( 1 , 6 ),  ( 2 , 2 ),  ( 3 , 1 ),  ( 4 , 3 ),  ( 5 , 5 ),  ( 6 , 5 ),  ( 7 , 1 )\},$\hspace{.05\textwidth}$E_{12}  =\{ ( 1 , 3 ),  ( 2 , 3 ),  ( 3 , 5 ),  ( 4 , 2 ),  ( 5 , 4 ),  ( 6 , 2 ),  ( 7 , 1 )\},$\\ 

$E_{13}  =\{ ( 1 , 5 ),  ( 2 , 3 ),  ( 3 , 1 ),  ( 4 , 4 ),  ( 5 , 2 ),  ( 6 , 4 ),  ( 7 , 6 )\},$\hspace{.05\textwidth}$E_{14}  =\{ ( 1 , 1 ),  ( 2 , 6 ),  ( 3 , 6 ),  ( 4 , 6 ),  ( 5 , 5 ),  ( 6 , 2 ),  ( 7 , 6 )\},$\\ 

$E_{15}  =\{ ( 1 , 2 ),  ( 2 , 3 ),  (3 , 4 ),  ( 4 , 1 ), ( 5 , 5 ),  ( 6 , 3 ),  ( 7 , 7 ),  \},$\hspace{.05\textwidth}$E_{16}  =\{ ( 1 , 4 ),  ( 2 , 1 ),  ( 3 , 4 ),  ( 4 , 2 ),  ( 5 , 3 ),  ( 6 , 5 ),  ( 7 , 6 )\},$\\ 

$E_{17}  =\{ (1 , 2 ),  ( 2 , 5 ),  ( 3 , 4 ),  ( 4 , 6 ),  ( 5 , 2 ),  ( 6 , 4 ),  ( 7 , 1 ) \},$\hspace{.05\textwidth}$E_{18}  =\{ ( 1 , 3 ),  ( 2 , 6 ),  ( 3 , 4 ),  ( 4 , 3 ),  ( 5 , 1 ),  ( 6 , 4 ),  ( 7 , 3 )\},$\\ 

$E_{19}  =\{ ( 1 , 3 ),  ( 2 , 1 ),  ( 3 , 1 ),  ( 4 , 6 ),  ( 5 , 4 ),  ( 6 , 3 ),  ( 7 , 4 )\},$\hspace{.05\textwidth}$E_{20} =\{ ( 1 , 4 ),  ( 2 , 3 ),  ( 3 , 1 ),  ( 4 , 3 ),  ( 5 , 2 ),  ( 6 , 2 ),  ( 7 , 5 )\},$\\ 

$E_{21}  =\{ ( 1 , 1 ),  ( 2 , 3 ),  ( 3 , 1 ),  ( 4 , 3 ),  ( 5 , 6 ),  ( 6 , 4 ),  ( 7 , 6 ) \},$\hspace{.05\textwidth}$E_{22}  =\{ ( 1 , 4 ),  ( 2 , 6 ),  ( 3 , 2 ),  ( 4 , 1 ),  ( 5 , 4 ),  ( 6 , 4 ),  ( 7 , 1 )\},$\\ 
}
\end{center}

The fact that $\mathcal{H'}$ has a covering number of $6$ can be checked quickly using a computer program by total enumeration.\\ 

\begin{lemma}
  \label{tau_6}
  $\tau(\mathcal{H'}) = 6$
\end{lemma}

The existence of $\mathcal{H'}$ and Lemma~\ref{tau_6} allows us to prove Theorem~\ref{r7_ext}.

\section{Concluding remarks}
In this paper we focused on constructing intersecting $r$-partite hypergraphs with $\tau(\mathcal{H})=r-1$. At the moment, for large $r$, the only constructions of such hypergraphs for large $r$ come from removing a vertex from a projective plane. Since projective planes only exist for prime powers, there are some values of $r$ for which we do not know if an extremal hypergraph for Ryser's Conjecture exists.

It would be of great interest to construct new examples of hypergraphs with $\tau(\mathcal{H})=r-1$, particularly for large $r$. To this end it would be interesting to even find hypergraphs for which $\tau(\mathcal{H})$ is ``close'' to $r-1$. Notice that from the projective plane construction, for \emph{every} $r$ it is possible to construct an $r$-partite intersecting hypergraphs with $\tau(\mathcal{H})=r-o(r)$, where $o(r)/r\to 0$ as $r \to \infty$. Indeed if for some $r$ there exists an $r$-partite intersecting hypergraph $\mathcal H$ with cover number $\tau$, then there are also $s$-partite intersecting hypergraphs with cover number $\tau$ for ever $s\geq r$ (these are constructed from $\mathcal H$ simply by adding $s-r$ new vertices to each edge). Therefore to construct hypergraphs with $\tau(\mathcal{H})=r-o(r)$ it is sufficient to know that for every $\epsilon>0$, there is an $N$ such that for all $n>N$ there is a prime power between $n$ and $(1+\epsilon)n$. In fact, there is always a prime in this interval for sufficiently large $n$. This can be shown using the Prime Number Theorem as an easy exercise. 

Any family of graphs satisfying $\tau(\mathcal{H})=r-o(r)$ which is different from the projective plane construction would already be interesting. We set the following problem to motivate further research.\\

\begin{problem}
For some fixed constant $c$ and every $r$ construct an $r$-uniform $r$-partite intersecting hypergraph with $\tau(\mathcal{H})=r-c$.
\end{problem}

In this paper we were interested in constructing extremal hypergraph for Ryser's Conjecture which had \emph{as few edges as possible}. Mansour, Song and Yuster conjectured that such hypergraphs have linearly many edges.\\

\begin{conjecture}[Mansour, Song, and Yuster, \cite{mans}]
Let $f(r)$ be the smallest integer for which there exists an $r$-uniform $r$-partite intersecting hypergraph  with $f(r)$ edges and $\tau(\mathcal{H})=r-1$.
Then $f(r)=\Theta(r)$.
\end{conjecture}

The first non-trivial lowerbound on $f(r)$ was proved in \cite{mans}, while the current best lowerbound is $f(r) > 3.052r + O(1)$ proved in \cite{Wanless}.

\section*{Acknowledgment}
The authors wish to thank Penny Haxell for giving an interesting talk on Ryser's Conjecture at the LSE which motivated them to research the conjecture. The first author also wishes to thank his supervisors Bernhard von Stengel and Jan van den Heuvel for helpful advice and discussions.

\bibliographystyle{plain}
\bibliography{ryser_bib}

\newpage
\appendix
\section{Appendix: Proof of Lemma~\ref{6_is_extremal}}

To make it easier to verify the claims in the following proof, Table~\ref{DTable} provides another representation of the hypergraph $\mathcal{H}$ referred to in Lemma~\ref{6_is_extremal}. Table~\ref{DTable} presents $\mathcal{H}$ in terms of its degree structure, where we use the notation $E(v)$ to denote the set of hyperedges in $\mathcal{H}$ that contain the vertex $v$. Each row in Table \ref{DTable} corresponds to a partition of $\mathcal{H}$, and the columns break downs the vertices in a given partition according to their degrees.

  \begin{table}[!htp]

    \footnotesize
  \begin{tabular}{c|c|c|c|c|}

    \cline{2-5}
    & \multicolumn{4}{c|}{Degrees} \\

    \cline{2-5}
    & $deg\ 1$ & $deg\ 2$ & $deg\ 3$ & $deg\ 4$ \\
    \cline{1-5}

    %Partition 1
    \multicolumn{1}{|c|}{\multirow{3}{*}{\scriptsize{Partition $1$}}} & \multirow{3}{*}{\shortstack{\scriptsize{$E((1,1)) = \{E_1\}$} \\ \scriptsize{$E((1,2)) = \{E_2\}$} \\ \scriptsize{$E((1,3)) = \{E_{3}\}$}}} & \multirow{3}{*}{\scriptsize{$E((1,4)) = \{E_4, E_5\}$}} &  & \multirow{3}{*}{\shortstack{\scriptsize{$E((1,5)) = \{E_6, E_7, E_8, E_9\}$} \\ \scriptsize{$E((1,6)) = \{E_{10}, E_{11}, E_{12}, E_{13}\}$}}} \\
    \multicolumn{1}{|c|}{} & & & &  \\
    \multicolumn{1}{|c|}{} & &  &  &  \\
    \cline{1-5}
    
    % Partition 2
    \multicolumn{1}{|c|}{\multirow{3}{*}{\scriptsize{Partition $2$}}} & \multirow{3}{*}{\scriptsize{$E((2,1)) = \{E_4\}$}} &  \multirow{3}{*}{\shortstack{\scriptsize{$E((2,2)) = \{E_6, E_{10}\}$} \\ \scriptsize{$E((2,3)) = \{E_9, E_{13} \}$}}} &  & \multirow{3}{*}{\shortstack{\scriptsize{$E((2,4)) = \{E_1, E_3, E_8, E_{11}\}$} \\ \scriptsize{$E((2,5)) = \{E_2, E_5, E_7, E_{12}\}$}}} \\
    \multicolumn{1}{|c|}{} & & & & \\
    \multicolumn{1}{|c|}{} & & & & \\
    \cline{1-5}
    
    % Partition 3
    \multicolumn{1}{|c|}{\multirow{3}{*}{\scriptsize{Partition $3$}}} & \multirow{3}{*}{\scriptsize{$E((3,1)) = \{E_7\}$}} &  \multirow{3}{*}{\shortstack{\scriptsize{$E((3,2)) = \{E_2, E_{11}\}$} \\ \scriptsize{$E((3,3)) = \{E_8, E_{13} \}$}}} &  & \multirow{3}{*}{\shortstack{\scriptsize{$E((3,4)) = \{E_1, E_5, E_9, E_{10}\}$} \\ \scriptsize{$E((3,5)) = \{E_3, E_4, E_6, E_{12}\}$}}} \\
    \multicolumn{1}{|c|}{} & & & & \\
    \multicolumn{1}{|c|}{} & & & & \\
    \cline{1-5}

    % Partition 4
    \multicolumn{1}{|c|}{\multirow{3}{*}{\scriptsize{Partition $4$}}} & \multirow{3}{*}{\scriptsize{$E((4,1)) = \{E_{12}\}$}} & \multirow{3}{*}{\scriptsize{$E((4,2)) = \{E_5, E_8\}$}} & \multirow{3}{*}{\shortstack{\scriptsize{$E((4,3)) = \{E_3, E_7, E_{10}\}$} \\ \scriptsize{$E((4,4)) = \{E_4, E_9, E_{11}\}$}}} &  \multirow{3}{*}{\scriptsize{$E((4,5)) = \{E_1, E_2, E_6, E_{13}\}$}} \\
    \multicolumn{1}{|c|}{} & \multicolumn{1}{|c|}{} &  &  &  \\
    \multicolumn{1}{|c|}{} & \multicolumn{1}{|c|}{} &  &  &  \\
    
    \cline{1-5}

    % Partition 5
    \multicolumn{1}{|c|}{\multirow{3}{*}{\scriptsize{Partition $5$}}} & \multirow{3}{*}{\scriptsize{$E((5,1)) = \{E_6\}$}} & \multirow{3}{*}{\scriptsize{$E((5,2)) = \{E_7, E_{11}\}$}} & \multirow{3}{*}{\shortstack{\scriptsize{$E((5,3)) = \{E_1, E_9, E_{12}\}$} \\ \scriptsize{$E((5,4)) = \{E_3, E_5, E_{13}\}$}}} &  \multirow{3}{*}{\scriptsize{$E((5,5)) = \{E_2, E_4, E_8, E_{10}\}$}} \\
    \multicolumn{1}{|c|}{} & \multicolumn{1}{|c|}{} &  &  &  \\
    \multicolumn{1}{|c|}{} & \multicolumn{1}{|c|}{} &  &  &  \\
    \cline{1-5}

    % Partition 6
    \multicolumn{1}{|c|}{\multirow{3}{*}{\scriptsize{Partition $6$}}} & \multirow{3}{*}{\scriptsize{$E((6,1)) = \{E_{10}\}$}} & \multirow{3}{*}{\scriptsize{$E((6,2)) = \{E_8, E_{12}\}$}} & \multirow{3}{*}{\shortstack{\scriptsize{$E((6,3)) = \{E_2, E_3, E_9\}$} \\ \scriptsize{$E((6,4)) = \{E_5, E_6, E_{11}\}$}}} &  \multirow{3}{*}{\scriptsize{$E((6,5)) = \{E_1, E_4, E_7, E_{13}\}$}} \\
    \multicolumn{1}{|c|}{} & \multicolumn{1}{|c|}{} &  &  &  \\
    \multicolumn{1}{|c|}{} & \multicolumn{1}{|c|}{} &  &  &  \\

    \cline{1-5}
    
  \end{tabular}
  \caption{Degree structure of $\mathcal{H}$}
  \label{DTable}
  
\end{table}

\begin{proof}[Proof of Lemma~\ref{6_is_extremal}]
      We first observe that if we exclude the hyperedge $E_1$ from $\mathcal{H}$ then the remaining hyperedges $\mathcal{H}$ form a \emph{linear} hypergraph. A hypergraph $\mathcal{G}$ is linear if the pairwise intersection of any two hyperedges in $\mathcal{G}$ is a singleton set.\\
    
    \begin{claim}
      \label{linear}
      For all $E_i, E_j \in E(\mathcal{H})$ such that $i, j \in \{2, \ldots, 13\}$ and i $\neq j$ we have that $|E_i \cap E_j| = 1$\\
    \end{claim}

    On the other hand the hyperedge $E_1$ intersect only two hyperedges of $\mathcal{H}$ more than once.\\

    \begin{claim}
      \label{regular}
      $|E_1 \cap E_9| = |E_1 \cap E_{13}| = 2$ and $|E_1 \cap E_i| = 1$ for all $E_i \in E(\mathcal{H}), i \not\in \{9, 13\}$\\
    \end{claim} 

    Furthermore, we observe that some of the hyperedges in $\mathcal{H}$ form a $2$-regular sub-hypergraph of $\mathcal{H}$.\\

    \begin{claim}\label{S1234}
      Let $S_1 = \{E_1, E_2, E_3, E_4, E_5\}$, $S_2 = \{E_{4}, E_{6}, E_{9}, E_{10}, E_{13}\}$ and $S_3 = \{E_2, E_7, E_8, E_{11}, E_{13}\}$, then $S_1$, $S_2$ and $S_3$ are all 2-regular linear sub-hypergraphs of $\mathcal{H}$, and thus we have $\tau(S_1) = \tau(S_2) = \tau(S_3) = 3$.\\
    \end{claim}

    We next show that if $\mathcal{H}$ has a cover $C$ that contains a vertex of degree $4$ then $|C| > 4$. From Table \ref{DTable} we can see that the partitions of $\mathcal{H}$ can be categorized into two types, those that contain two vertices of degree $4$ and those that only contain one vertex of degree $4$.

    In Claim~\ref{deg_4A} we will show that if $C$ contains a vertex of degree $4$ that is also from a partition with two vertices of degree $4$ then $|C|  > 4$. While in Claim~\ref{deg_4B} we will show that $|C| > 4$ if $C$ contains a vertex of degree $4$ that is from a partition that contains only one vertex of degree $4$.   \\

    \begin{claim}
      \label{deg_4A}
      If $C$ is a cover of $\mathcal{H}$ that contains a vertex $v$ of degree $4$, and $v$ is from a partition that contains two vertices of degree $4$, then $|C| > 4$.
    \end{claim}

    \begin{proof}
     Assume that $C$ is as in the claim, and that $C$ contains the vertex $(1,6)$. If $C$ is a cover of $\mathcal{H}$ that contains $(1,6)$ then if $C$ doesn't contain $(1,5)$ (the other vertex of degree $4$ in partition 1) then by Claim~\ref{linear} it must contain at least four more vertices to cover $E_6, E_7, E_8$ and $E_9$.

      Hence assume $C$ contains both $(1,6)$ and $(1,5)$. By Claim~\ref{S1234}, $C$ needs to contain three more vertices to cover the hyperedges in $S_1 = \{E_1, E_2, E_3, E_4, E_5\}$. Therefore if $C$ contains $(1,6)$ it will contain at least five vertices.
      
      The cases when $v$ is one of the vertices $(1,5), (2,4), (2,5)$ and $(3,4)$ are proved identically, replacing $S_1$ with $S_2$ or $S_3$ where necessary. This leaves the case $(3,5)$ where the above reasoning doesn't apply since in this case it is possible to cover the hyperedges not in $E((3,5))$ by three vertices (since $|E_1 \cap E_9| = |\{ (3,4), (5,3)\}| = 2$). However, if $C$ doesn't contain $(5,3)$ we can still apply the above reasoning to get $|C| > 4$. Therefore assume $C$ contains $(3,5)$ and $(5,3)$. In this situation $C$ must still contain three more vertices to cover the edges in $\mathcal{S}_3$ which concludes the proof. \\

    \end{proof}

    We now consider the covers of $\mathcal{H}$ that contain a vertex of degree $4$ and are in a partition that only contains one vertex of degree $4$. These vertices are $(4,5)$, $(5,5)$ and $(6,5)$.\\

    \begin{claim}
      \label{deg_4_int}
      Let $v$ and $u$ be two distinct vertices of $\mathcal{H}$ such that $v, u \in \{ (4,5), (5,5), (6,5)\}$ then $|E(v) \cap E(u)| \leq 7$.\\
    \end{claim}

    \begin{claim}
      \label{deg3_deg4}
      Let $v$ be a vertex from the set $\{(4,5),  (6,5)\}$ then the only vertices $w$ of degree $3$ such that $E(v) \cap E(w) = \emptyset$ are the vertices of degree $3$ in the same partition of $v$.

      While if $v$ is the vertex $(5,5)$ then the only vertices $w$ of degree $3$ such that $E(v) \cap E(w) = \emptyset$ are the vertices of degree $3$ in the same partition as $(5,5)$ and the vertex $(6,4)$.
      
    \end{claim}

    It is easy to see that Claim \ref{deg3_deg4} implies the following claim. 
    
    \begin{claim}
      \label{degree_3}
      Let $v$ and $u$ be two distinct vertices of $\mathcal{H}$ such that $v, u \in \{(4,5), (5,5), (6,5)\}$ and let $w$ be any vertex of degree $3$ in $\mathcal{H}$ except $(6,4)$, then this implies $|\big(E(v) \cup E(u) \big) \cap E(w)| \geq 1$.\\
    \end{claim}

    Note that $(6,4)$ is an exception in Claim \ref{degree_3} because $\big( E((5,5)) \cup E((6,5)) \big) \cap E((6,4)) = \emptyset$.\\
    
      \begin{claim}
        \label{deg_4B}
      If $C$ is a cover of $\mathcal{H}$ that contains one of the vertices $(4,5), (5,5)$ and $(5,6)$ then $|C| > 4$.\\
    \end{claim}

    \begin{proof}
      Let $C$ by a cover of $\mathcal{H}$ that contains one of the vertices in the set $\{(4,5), (5,5) ,(5,6)\}$ with $|C| \leq 4$. From Lemma \ref{deg_4A} it follows that if $C$ contains any of the vertices of degree $4$ that are not $(4,5), (5,5)$ and $(5,6)$, then $|C| \geq 5$, thus we can assume that $C$ doesn't contain any other vertex of degree $4$ that is not in the set $\{(4,5), (5,5), (5,6)\}$.
      
      Since $|E((4,5)) \cup E((5,5)) \cup E((6,5))| = |\{E_1, E_2, E_4, E_6, E_7, E_8, E_{10}, E_{13}\}| = 8$, if $C$ contains all three of $(4,5), (5,5)$ and $(5,6)$, it will need to contain at least two more vertices (since we excluded the possibility of it containing any more vertices of degree $4$), which will contradict $|C| \leq 4$. Therefore we can assume that $C$ contains some but not all of the three vertices $(4,5), (5,5)$ and $(6,5)$.

      Assume that $C$ contains exactly two distinct vertices $u$ and $v$ from the set $\{(4,5), (5,5), (5,6)\}$. By Claim~\ref{deg_4_int} $|E(u) \cap E(v)| \leq 7$. Since $C$ cannot contain any more vertices of degree $4$, it will need to contain at least two more vertices of degree $3$ to cover $\mathcal{H}$. However, from Claim~\ref{degree_3} we know that the only vertex of degree $3$ that cover three more hyperedges if included in $C$ with $v$ and $u$ is possibly $(6,4)$. This contradicts $C$ containing only four vertices. 
      
      Finally, we consider the case of $C$ containing only one vertex of degree four. Assume first that the only vertex of degree four contained in $C$ is $(4,5)$, then $C$ will need to contain at least three more vertices of degree $3$ to cover the rest of $\mathcal{H}$, moreover we need each vertex $w$ of these three vertices to satisfy the condition $w \cap E((4,5)) = \emptyset$. However, by Claim \ref{deg3_deg4} there is a maximum of only two vertices of degree $3$ that satisfy this condition, thus $(4,5)$ can't be the only vertex of degree $4$ in $C$. We can also see that the same reasoning applies to the case when the only vertex of degree $4$ contained in $C$ is $(6,5)$.

      The only remaining possibility is for the only vertex of degree $4$ contained in $C$ to be $(5,5)$. Again by the same reasoning as in the case $(4,5)$, and again by using Claim \ref{deg3_deg4} we conclude that the remaining three vertices in $C$ must be the vertices $(5,3), (5,4)$ and $(6,4)$. However, since $E((5,4)) \cap E((6,4)) \neq \emptyset$ this means that $(5,3), (5,4)$ and $(6,4)$ can't cover the nine remaining hyperedges in $\mathcal{H}$ that are uncovered by $(5,5)$, which contradicts $C$ being a cover $\mathcal{H}$. 

    \end{proof}

    Claim~\ref{deg_4B} and Claim~\ref{deg_4A} show that if $C$ is a cover of $\mathcal{H}$, and $|C| \leq 4$ then it cannot contain any vertex of degree four. However, four vertices of degree at most $3$ can cover at most $12$ hyperedges, which contradicts $C$ being a cover of $\mathcal{H}$.  

    \end{proof}

\end{document}